\documentclass[11pt, reqno]{amsart}

\usepackage{amsfonts, amssymb, amscd}
\usepackage{graphicx}
\usepackage{hyperref}
\usepackage{slashed}
\usepackage{fullpage}

\newtheorem{thm}{Theorem}[section]
\newtheorem{cor}[thm]{Corollary}
\newtheorem{lem}[thm]{Lemma}

\theoremstyle{remark}
\newtheorem{rem}[thm]{Remark}

\theoremstyle{definition}
\newtheorem{defin}[thm]{Definition}

\newcommand{\uA}{\underline{A}}

\newcommand{\R}{\mathbb{R}}
\renewcommand{\H}{\mathbb{H}}
\newcommand{\C}{\mathbb{C}}
\newcommand{\Sp}{\mathbb{S}}

\renewcommand{\Re}{\mathrm{Re}}
\renewcommand{\Im}{\mathrm{Im}}

\title{An unconstrained {L}agrangian formulation and conservation laws for the {S}chr\"odinger map system}

\author[Paul Smith]{Paul Smith}
\address{University of California, Berkeley}
\email{smith@math.berkeley.edu}

\thanks{The author was supported by NSF grant DMS-1103877.}

\begin{document}

\begin{abstract}
We consider energy-critical Schr\"odinger maps from $\R^2$ 
into $\Sp^2$ and $\H^2$.
Viewing such maps with respect to orthonormal frames on the pullback bundle provides
a gauge field formulation of the evolution. 
We show that this gauge field system is the set of Euler-Lagrange equations corresponding to an action 
that includes a Chern-Simons term.
We also introduce the stress-energy tensor and derive conservation laws.
In conclusion we offer comparisons between Schr\"odinger maps and the closely related
Chern-Simons-Schr\"odinger system.
\end{abstract}

\maketitle

%\tableofcontents

\section{Introduction}

Our main purpose in this article is to derive the Schr\"odinger maps system 
at the level of the differentiated, gauged system using a variational approach. 
In part this is to provide a basis for resolving a certain tension:
Schr\"odinger maps are usually introduced as constrained geometric evolution equations, 
whereas state-of-the-art results on Schr\"odinger maps are proved at the level
of the gauged system, with little if any reference to the underlying map.
These two formulations are related in a simple way. 
The gauged system is obtained by representing the differentiated 
Schr\"odinger maps system with respect to a space and time dependent orthonormal frame.
Using the Frobenius theorem, one recovers the Schr\"odinger map system from
the gauged system.
In spite of this close relationship, certain gaps have persisted in what might be called the
dictionary that translates between these two formulations.
In particular, at the level of maps, the equation is easily seen to be Hamiltonian, though
the variational formulation is not entirely satisfactory thanks to topological obstructions.
At the level of the differentiated system, topological obstructions cease to exist, though different
difficulties emerge, and only partial Hamiltonian and variational descriptions were known.
In this article we fill in this gap, providing a natural variational formulation and Hamiltonian.

In particular we study
the energy-critical Schr\"odinger map system with target $\Sp^2$ or target $\H^2$.
Our first result is a natural variational formulation of the differentiated, gauged system. 
That is, in \S \ref{Sec:Lagrangian} we introduce the action.
Next, in \S \ref{Sec:ConservationLaws}, we introduce a natural stress-energy tensor and derive conservation laws.
It is here that we introduce the Hamiltonian, as it may be rewritten in a simple way in terms of the stress-energy tensor. 
In \S \ref{Sec:CSS}, we take up comparing Schr\"odinger maps with the Chern-Simons-Schr\"odinger system, which is
suggested in part by a shared Chern-Simons term in their actions.
Finally, we consider in the Appendix gradient flow and solitons from the gauged point of view. These objects
are not only interesting in their own right but also are important because they are needed to construct the caloric gauge.

\subsection{Geometric map equations}

Suppose we have $\phi : \R^d \to M$, where $\R^d$ is Euclidean space, $M$ is a Riemannian manifold with metric $h$, and $\phi$ is a smooth map.
Consider the Lagrangian
\begin{equation} \label{Lag}
\frac12 \int_{\R^d} \langle \partial_j \phi, \partial_j \phi \rangle_{h(\phi(x))} dx
\end{equation}
where here and throughout we sum repeated Roman indices over all spatial variables.
The associated Euler-Lagrange equation is
\begin{equation} \label{HM}
(\phi^* \nabla)_j \partial_j \phi = 0
\end{equation}
the solutions of which are called \emph{harmonic maps}. 
Here $\nabla$ denotes the Levi-Civita connection on $M$ and $\phi^* \nabla$ denotes the pullback of this connection to $\R^d$.
The downward gradient flow associated to \eqref{Lag} generates the
\emph{harmonic map heat flow} equation
\begin{equation} \label{HMHF}
\partial_t \phi = (\phi^* \nabla)_j \partial_j \phi
\end{equation}
If the target manifold $M$ is K\"ahler with complex structure $J$,
then to derive a Schr\"odinger evolution variationally we need to introduce in the action a suitable term.
As this term ought only to carry one derivative, the natural pairing is with a 1-form.
A drawback of this Lagrangian formulation is that there can
be topological obstructions to global nonvanishing 1-forms, such as is the case
with $\Sp^2$. 
This particular case may be handled by first stereographically projecting to $\C$ and then 
on that level writing down a suitable action \cite{MaPaSo94, GrSt02}, though this procedure does not
genuinely circumvent the fundamental topological issue.
In any case, at the level of maps we are led to
the \emph{Schr\"odinger map} equation
\begin{equation} \label{SM}
\partial_t \phi = J(\phi) (\phi^* \nabla)_j \partial_j \phi
\end{equation}
Equation \eqref{SM} arises in ferromagnetism
as a Heisenberg model for a ferromagnetic spin system
and describes the classical spin
\cite{La67, PaTo91, MaPaSo94, ChShUh00, NaStUh03}.

Solutions of \eqref{HMHF} and \eqref{SM} are preserved by the rescalings
\[
\phi(t, x) \mapsto \phi(\lambda^2 t, \lambda x) \quad \quad \lambda > 0
\]
and solutions of \eqref{HM} are preserved by such scalings in the spatial variable.
For each of these equations, the natural energy is given by \eqref{Lag},
which also obeys a scaling law:
\[
E(\phi) := \frac12 \int_{\R^d} \langle \partial_j \phi, \partial_j \phi \rangle_{h(\phi(x))} dx,
\quad \quad
E(\phi(x)) = \lambda^{2-d} E(\phi(\lambda x))
\]
Energy is formally conserved by \eqref{SM}, and as noted the flow of \eqref{HMHF} is
the downward gradient flow associated to the energy.
In dimension $d = 2$, both the energy and the equations are preserved by rescalings, and for
this reason this is called the energy-critical setting. From now on we assume $d = 2$.

\subsection{Gauges}

One theme unifying the study of equations \eqref{HM}---\eqref{SM}
is the use of \emph{gauges} or \emph{moving frames}:
for each point in the domain, e.g.~each $(t, x) \in I \times \R^2$ in cases \eqref{HMHF} and \eqref{SM},
we choose an orthonormal basis of $TM_{\phi(t, x)}$.
Frames have been used extensively in studying harmonic maps \cite{He91}, and
their use in the setting of Schr\"odinger maps in proving wellposedness was initiated in \cite{ChShUh00}.
Our notation and perspective follow closely that in \cite[Chapter 6]{Tao06}.
In the energy-critical case with a surface as the target, we have one degree of freedom in our choice
of orthonormal frame for each $(t, x)$. 
For maps from $\R^2$ into $M \in \{\Sp^2, \H^2\}$ evolving on some time interval $I$, 
a gauge choice may be represented by the diagram 
\[
\begin{CD}
(I \times \R^2) \times \C @>e>> \phi^* TM @>>> TM \\
@AA\psi_\alpha A @AA\partial_\alpha \phi A @VV\pi V\\
I \times \R^2 @>id>> I \times \R^2 @>\phi >> M
\end{CD}
\]
Here, for each $\alpha \in \{0, 1, 2\}$, the partial derivative $\partial_\alpha \phi$
is interpreted as a section of the pullback bundle $\phi^* TM$, and each $\psi_\alpha$
is interpreted as a section of the pullback bundle $e^* \phi^* TM$.
Because the underlying manifold $M$ is complex, we complexify the tangent spaces
comprising the pullback bundles.
The map $e$ is identified with a choice of orthonormal frame on $\phi^* TM$ in the following way.
For each $(t, x) \in I \times \R^2$, choose an orthonormal
basis $\{v(t, x), J(\phi(t, x)) v(t, x)\}$ of $TM_{\phi(t, x)}$. 
Then let
$e(t, x) : \C \to TM_{\phi(t, x)}$ denote the linear transformation that acts
according to $z \mapsto \Re(z) v(t, x) + \Im(z) J(\phi(t, x)) v(t, x)$.
Through $e$, the Levi-Civita connection
pulls back to the covariant derivatives $D_\alpha := \partial_\alpha + i A_\alpha$, which generate curvatures
$F_{\alpha \beta} := \partial_\alpha A_\beta - \partial_\beta A_\alpha$. 
Orthonormality of the frame ensures $A_\alpha \in \R$.
The zero-torsion property of the connection enforces the compatibility condition $D_\alpha \psi_\beta = D_\beta \psi_\alpha$. 
Using the fact that $\Sp^2$ has constant curvature $+1$, one may calculate directly that 
$F_{\alpha \beta} = \Im(\bar{\psi}_\beta \psi_\alpha)$.
Similarly, the constant $-1$ curvature of $\H^2$ leads to
$F_{\alpha \beta} = -\Im(\bar{\psi}_\beta \psi_\alpha)$.
So that we can consider both cases
simultaneously, we write $F_{\alpha \beta} = \mu \Im(\bar{\psi}_\beta \psi_\alpha)$, taking $\mu = +1$
for the sphere and $\mu = -1$ for the hyperbolic plane.
Thus for any map $\phi$ and any choice of frame $e(t, x)$, it holds that
\[
F_{\alpha \beta} = \mu \Im(\bar{\psi}_\beta \psi_\alpha) \quad \quad \text{and} \quad \quad
D_\alpha \psi_\beta = D_\beta \psi_\alpha
\]
These relations are all preserved by the transformations
\begin{equation}\label{gauge-freedom}
\phi \mapsto e^{-i \theta} \phi
\quad \quad
A \mapsto A + d \theta
\end{equation}
where $\theta(t, x)$ is a compactly supported real-valued function (we only use
time-independent functions in the case of \eqref{HM-gauge}).
This gauge invariance corresponds to the freedom we have in the choice of frame $e(t, x)$.

Here and throughout we use $\partial_0$ and $\partial_t$ interchangeably.
We also adopt the convention that Greek indices are allowed to assume values
from the set $\{0, 1, 2\}$, whereas Roman indices are restricted to $\{1, 2\}$, meaning
that Roman indices indicate only spatial variables. 
Our summation conventions are that repeated Roman indices are summed over $\{1, 2\}$
and that repeated Greek indices are summed over $\{0, 1, 2\}$.

At the gauge field level, the energy-critical harmonic maps equation \eqref{HM} assumes the form
\begin{equation} \label{HM-gauge}
\begin{cases}
0 &= D_j \psi_j \\
F_{12} &= \mu \Im(\bar{\psi}_2 \psi_1) \\
D_1 \psi_2 &= D_2 \psi_1
\end{cases}
\end{equation}
The procedure for obtaining gauge field representations of evolution equations is slightly less
straightforward. For the harmonic map heat flow, for instance, we begin by pulling back the left
and right hand sides of equation \eqref{HMHF}:
\begin{equation} \label{HMHF-gauge0}
\psi_t = D_j \psi_j
\end{equation}
To obtain an evolution equation from \eqref{HMHF-gauge0}, we covariantly differentiate in a spatial
direction by applying $D_k$ and then invoke the compatibility condition $D_k \psi_t = D_t \psi_k$:
\[
D_t \psi_k = D_k D_j \psi_j
\]
By using the curvature relation to commute $D_k$ and $D_j$ and then applying the compatibility condition
$D_j \psi_k = D_k \psi_j$, we obtain a covariant heat equation for $\psi_k$.
All told, we arrive at the system
\begin{equation} \label{HMHF-gauge}
\begin{cases}
D_t \psi_k &= D_j D_j \psi_k - i F_{jk} \psi_j \\
F_{01} &= \mu \Im(\bar{\psi}_1 D_j \psi_j) \\
F_{02} &= \mu \Im(\bar{\psi}_2 D_j \psi_j) \\
F_{12} &= \mu \Im(\bar{\psi}_2 \psi_1) \\
D_1 \psi_2 &= D_2 \psi_1
\end{cases}
\end{equation}
Note that we have eliminated the field $\psi_t$.
The gauge field equations for Schr\"odinger maps are similarly derived. For
\eqref{SM}, the analogue of \eqref{HMHF-gauge0} is
\[
\psi_t = i D_j \psi_j
\]
and we arrive at the system
\begin{equation} \label{SM-gauge}
\begin{cases}
D_t \psi_k &= i D_j D_j \psi_k + F_{jk} \psi_j \\
F_{01} &= \mu \Re(\bar{\psi}_1 D_j \psi_j) \\
F_{02} &= \mu \Re(\bar{\psi}_2 D_j \psi_j) \\
F_{12} &= \mu \Im(\bar{\psi}_2 \psi_1) \\
D_1 \psi_2 &= D_2 \psi_1
\end{cases}
\end{equation}
\begin{rem}
All three of the above systems, i.e., \eqref{HM-gauge}, \eqref{HMHF-gauge}, and \eqref{SM-gauge},
are preserved by gauge transformations \eqref{gauge-freedom}.
In order to obtain well-defined flows, one must eliminate the gauge freedom by making a gauge
choice. See \cite[Chapter 6]{Tao06} for a survey of various gauge choices.
It appears that the best gauge for handling arbitrary Schr\"odinger maps (e.g., maps
without any symmetry assumption) is the caloric gauge,
which was introduced in \cite{Tao04} in the context of wave maps and first applied to
Schr\"odinger maps in \cite{BeIoKeTa11a}. The preferred gauge for studying 
Schr\"odinger maps with equivariant symmetry
and harmonic maps is the Coulomb gauge.
\end{rem}

\begin{rem} \label{Return}
It is natural to ask whether solutions of \eqref{HM-gauge}, \eqref{HMHF-gauge}, or \eqref{SM-gauge}
must arise from an underlying map. Assuming sufficient decay and regularity, this is indeed the case.
We demonstrate this for Schr\"odinger maps into $\Sp^2 \hookrightarrow \R^3$, where the embedding
is the usual one.
Let $\phi$ be a Schr\"odinger map and let $e_1, e_2$ denote the two vectors of the orthonormal frame $e$.
Define for $\alpha = 0, 1, 2$,
\[
\Phi = \begin{bmatrix} e_1 & e_2 & \phi \end{bmatrix},
\quad \quad
R_\alpha =
\begin{bmatrix}
0 & -A_\alpha & \Re(\psi_\alpha) \\
A_\alpha & 0 &  \Im(\psi_\alpha) \\
-\Re(\psi_\alpha) & -\Im(\psi_\alpha) & 0
\end{bmatrix}
\]
Then using \eqref{SM} and the definitions, we find that $\Phi, R$ satisfy the Mayer-Lie system
\begin{equation} \label{recover-eq}
\partial_\alpha
\Phi
=
\Phi R_\alpha \quad \quad \alpha = 0, 1, 2
\end{equation}
That \eqref{recover-eq} satisfies the Frobenius integrability condition may be described most succinctly
using the Maurer-Cartan 1-form $\omega = \Phi^{-1} d\Phi$, which satisfies
\begin{equation} \label{MC}
d \omega + \frac12 [\omega, \omega] = 0
\end{equation}
In this perspective $\Phi$ is interpreted as an element of $SO(3)$ and each 
$R_\alpha$ as an element of the corresponding Lie algebra $so(3)$.

To obtain $\Phi$ from $R$, we reverse the argument: if $(\psi, A)$ satisfy \eqref{SM-gauge}, then
the integrability condition \eqref{MC} is satisfied with $\omega = R_\alpha dx^\alpha$. If $(\psi, A)$ are rapidly decaying, then 
we can specify a (uniform) boundary condition for $\Phi$ at spatial infinity and 
recover $\Phi$ at all points by integrating in from infinity. 
If we have special structure such as equivariance, then instead we can specify $\Phi$ at a point $x \in \R^2$ and integrate out.

Analogous statements hold for $\H^2$ embedded in $\R^3$ endowed with the Minkowski metric. In that setting $\Phi$
is interpreted as an element of the Lorentz group $SO(2, 1)$ and the $R_\alpha$ as elements of the associated Lie algebra $so(2, 1)$.
For more details in the $\H^2$ setting and for additional related comments, see \cite[\S 2]{Tao04}.
\end{rem}

\subsection{Topology}

\begin{defin}
The \emph{charge} $c_1$ of a vector bundle over $\R^2$ with connection $A$ is the integral
\[
c_1 := \frac{1}{2\pi} \int_{\R^2} d\uA = \frac{1}{2\pi} \int_{\R^2} F_{12} dx^1 \wedge dx^2
\]
The charge is also known as the first Chern number.
\end{defin}
The ``underline" notation introduced here we will also use in the sequel:
an underlined form means that we take only the spatial components of that form.
For instance, if $A = A_0 dt + A_j dx^j$, then $\uA = A_j dx^j$.
\begin{lem}
For rapidly decaying solutions of \eqref{HMHF-gauge} or \eqref{SM-gauge},
charge is conserved, i.e.,
\[
\partial_t \frac{1}{2\pi} \int_{\R^2} F_{12} dx = 0
\]
\end{lem}
\begin{proof}
Because $d^2 A = 0$ for any 1-form $A$,
\begin{equation} \label{geocon}
\partial_t F_{12} - \partial_1 F_{02} + \partial_2F_{01} = 0
\end{equation}
\end{proof}
A less obvious fact is that for the system \eqref{HM-gauge}, charge is \emph{quantized}, 
which is to say that it is integer-valued. At the level of maps, this follows
from the Gauss-Bonnet theorem and the fact that $d\uA$ is the pullback by the map
of the volume form on the target. Charge in fact characterizes the homotopy class.
To prove quantization at the gauge field level, one may exhaust $\R^2$ with nested discs,
apply Stokes' theorem to the integral of $d\uA$ over each disc, 
and then control the resulting integrals that arise on the boundary.
The field equations are of course essential in establishing quantization.
See \cite[Chapter 3]{MaSu04} for further discussion.

\section{Lagrangian formulation} \label{Sec:Lagrangian}

In this section we show that the system \eqref{SM-gauge} arises as the Euler-Lagrange
equations of a suitable gauge-invariant action. The difficulties encountered at the level
of the map do not arise. In view of Remark \ref{Return}, this furnishes a Lagrangian formulation
for the Schr\"odinger map system.
In carrying out variations we work formally, assuming smoothness of all quantities
and assuming that fields and variations are rapidly decaying.
\begin{thm} \label{thm:SchLag}
The energy-critical gauged Schr\"odinger map system \eqref{SM-gauge} is generated by the action
\[
\begin{split}
L_{Sch}(\psi, A) 
:=& \int_{\R^{2+1}} 
\left[ \Re(\bar{\psi}_2 D_t \psi_1) - \Im(\overline{D_j \psi_2} D_j \psi_1) \right] dx^1 \wedge dx^2 \wedge dt \\
& + \frac12 \int_{\R^{2+1}} (|\psi_1|^2 + |\psi_2|^2) dt \wedge dA +
\mu \frac12 \int_{\R^{2+1}} A \wedge dA
\end{split}
\]
provided that the compatibility condition $D_1 \psi_2 = D_2 \psi_1$ holds at the initial time.
\end{thm}
\begin{proof}
We verify the claim by calculating the variation.

\noindent \textbf{Variation of $\psi$}.
The variation of $\psi_1, \psi_2$ respectively give rise to the $D_t \psi_2$ and $D_t \psi_1$ evolutions
of \eqref{SM-gauge}.
Under the variation $\psi_1 \mapsto \psi_1 + \varepsilon \phi$, the terms linear in $\varepsilon$
from 
\[
\Re(\bar{\psi}_2 D_t \psi_1), \quad -\Im(\overline{D_j \psi_2} D_j \psi_1), \quad \frac12 F_{12} |\psi_1|^2,
\]
are, respectively,
\[
\Re(\bar{\psi}_2 D_t \phi), \quad
- \Im(\overline{D_j \psi_2} D_j \phi), \quad
F_{12} \Re(\bar{\psi}_1 \phi)
\]
Integrating by parts in
\[
\int_{\R^{2+1}} \left[ \Re(\bar{\psi}_2 D_t \phi) - \Im(\overline{D_j \psi_2} D_j \phi) + F_{12} \Re(\bar{\psi}_1 \phi) \right] dx dt
\]
yields
\[
\int_{\R^{2+1}} \left[ -\Re(\bar{\phi} D_t \psi_2) - \Im(\bar{\phi} D_j D_j \psi_2) + F_{12} \Re(\bar{\phi} \psi_1) \right] dx dt
\]
which leads to the evolution equation
\[
D_t \psi_2 = i D_j D_j \psi_2 + F_{12} \psi_1
\]
Similarly, under the variation $\psi_2 \mapsto \psi_2 + \varepsilon \phi$ we obtain
the $\varepsilon$-linear terms
\[
\Re(\bar{\phi} D_t \psi_1), \quad
-\Im(\overline{D_j \phi} D_j \psi_1), \quad
F_{12} \Re(\bar{\phi} \psi_2)
\]
which lead to the evolution equation
\[
D_t \psi_1 = i D_j D_j \psi_1 - F_{12} \psi_2
\]

\noindent \textbf{Variation of $A$}. The variation of $A_0$ leads to the $F_{12}$ curvature equation.
Varying $A_1$ and $A_2$ respectively yield preliminary $F_{02}$ and $F_{01}$ equations.
To obtain the compatibility condition $D_1 \psi_2 = D_2 \psi_1$, we enforce it
at time zero and then show using Gronwall's inequality that the condition persists.
Once we have the compatibility condition,
we can substitute it back into the preliminary $F_{0j}$ equations to obtain the equations
appearing in \eqref{SM-gauge}.

Under the variation $A \to A + \varepsilon B$, we get from $\mu \frac12 \int A \wedge dA$ the $\varepsilon$-linear term
\[
\mu B \wedge dA
\]
which can be verified using Stokes and the fact that for 1-forms $A, B$ we have
$d(A \wedge B) = dA \wedge B - A \wedge dB$.
Upon expansion, the term appears as
\begin{equation} \label{AdA-variation}
\mu \int_{\R^{2+1}} \left( B_t F_{12} - B_1 F_{02} + B_2 F_{01} \right) dx dt
\end{equation}
From $\Re(\bar{\psi}_2 D_t \psi_1)$, we get from the variation of $A$ the $\varepsilon$-linear term
\begin{equation} \label{At-variation}
- B_t \Im(\bar{\psi}_2 \psi_1)
\end{equation}
As there are no other $A_t$ variation terms, we conclude from \eqref{AdA-variation}
and \eqref{At-variation} that
\begin{equation} \label{F12}
F_{12} = \mu \Im(\bar{\psi}_2 \psi_1)
\end{equation}
We also have $\varepsilon$ terms coming from the variation of the $A_j$.
In particular, $-\Im(\overline{D_j \psi_2} D_j \psi_1)$ contributes
\begin{equation} \label{var-2}
B_j \Re(\bar{\psi}_2 D_j \psi_1) - B_j \Re(\overline{D_j \psi_2} \psi_1)
\end{equation}
Finally, we have to handle
\[
\frac12 \int_{\R^{2+1}} (|\psi_1|^2 + |\psi_2|^2) dt \wedge dA
\]
To do so we first invoke Stokes to obtain
\[
\frac12 \int_{\R^{2 + 1}} d\left[ (|\psi_1|^2 + |\psi_2|^2) dt\right] \wedge A
\]
and then expand to get
\begin{equation} \label{Last-variation}
\frac12 \int_{\R^{2 + 1}} 
\left[ 
A_1 \partial_2 (|\psi_1|^2 + |\psi_2|^2)
- A_2 \partial_1 (|\psi_1|^2 + |\psi_2|^2)
\right] dx^1 \wedge dx^2 \wedge dt 
\end{equation}
Varying \eqref{Last-variation} with respect to $A$ and then expanding yields
the $\varepsilon$-linear terms
\begin{equation} \label{var-3}
\int_{\R^{2 + 1}} \left[ B_1 \Re(\bar{\psi}_1 D_2 \psi_1) + B_1 \Re(\bar{\psi}_2 D_2 \psi_2) 
- B_2 \Re(\bar{\psi}_1 D_1 \psi_1) - B_2 \Re(\bar{\psi}_2 D_1 \psi_2) \right] dx dt
\end{equation}
Comparing the $B_1$ terms in \eqref{AdA-variation}, \eqref{var-2}, and \eqref{var-3}
leads to
\[
\int \left[ \Re(\bar{\psi}_2 D_1 \psi_1) - \Re(\overline{D_1 \psi_2} \psi_1)
+ \Re(\bar{\psi}_1 D_2 \psi_1) + \Re(\bar{\psi}_2 D_2 \psi_2) - \mu F_{02} \right] = 0
\]
This yields
\begin{equation} \label{F02-prelim}
\mu F_{02} = \Re(\bar{\psi}_2 D_j \psi_j) + \Re(\bar{\psi}_1 (D_2 \psi_1 - D_1 \psi_2))
\end{equation}
Similarly, comparing $B_2$ terms leads to
\[
\int \left[ \Re(\bar{\psi}_2 D_2 \psi_1) - \Re(\overline{D_2 \psi_2} \psi_1) - \Re(\bar{\psi}_1 D_1 \psi_1) - \Re(\bar{\psi}_2 D_1 \psi_2) + \mu F_{01} \right] = 0
\]
and hence
\begin{equation} \label{F01-prelim}
\mu F_{01} = \Re(\bar{\psi}_1 D_j \psi_j) + \Re(\bar{\psi}_2 (D_1 \psi_2 - D_2 \psi_1))
\end{equation}
By direct calculation one may verify that \eqref{geocon} holds with 
\eqref{F12}, \eqref{F02-prelim}, and \eqref{F01-prelim}.

\noindent \textbf{The compatibility condition.}
Set
\[
\Theta := D_1 \psi_2 - D_2 \psi_1
\]
Then
\begin{equation} \label{DtTheta}
D_t \Theta = D_1 D_t \psi_2 - D_2 D_t \psi_1 + iF_{01} \psi_2 - i F_{02} \psi_1
\end{equation}
By direct calculation,
\[
D_t \psi_1 = i D_1 D_j \psi_j - i D_2 \Theta,
\quad \quad
D_t \psi_2 = i D_2 D_j \psi_j + i D_1 \Theta
\]
which, upon substituting into \eqref{DtTheta}, yield
\[
\begin{split}
D_t \Theta &= i(D_1 D_2 - D_2 D_1) D_j \psi_j + i D_j D_j \Theta + i F_{01} \psi_2 - i F_{02} \psi_1 \\
&= -F_{12} D_j \psi_j + i F_{01} \psi_2 - i F_{02} \psi_1 + i D_j D_j \Theta
\end{split}
\]
Invoking \eqref{F12}, \eqref{F02-prelim}, and \eqref{F01-prelim}, we find
\[
-F_{12} D_j \psi_j + i F_{01} \psi_2 - i F_{02} \psi_1
= \mu i \left[ \psi_2 \Re(\bar{\psi}_2 \Theta) + \psi_1 \Re(\bar{\psi}_1 \Theta) \right]
\]
Therefore
\[
D_t \Theta = i D_j D_j \Theta + \mu i \left[ \psi_2 \Re(\bar{\psi}_2 \Theta) + \psi_1 \Re(\bar{\psi}_1 \Theta) \right]
\]
so that in particular
\[
\Re(\bar{\Theta} D_t \Theta)
=
\partial_j \Re(\bar{\Theta} i D_j \Theta)
-
\mu \Im(\bar{\Theta} \left[ \psi_2 \Re(\bar{\psi}_2 \Theta) + \psi_1 \Re(\bar{\psi}_1 \Theta) \right])
\]
Consequently
\[
\partial_t \frac12  \int_{\R^2} |\Theta|^2 dx
\leq
\sup_{\R^2} \left( |\psi_1|^2 + |\psi_2|^2 \right) \int_{\R^2} |\Theta|^2 dx
\]
Therefore if $\Theta = 0$ at time zero, then we conclude by Gronwall's inequality that
$\Theta$ is zero for all later times for which the solution exists.
By time reversibility of the system, this means that the compatibility condition
\begin{equation} \label{compatibility}
D_1 \psi_2 = D_2 \psi_1
\end{equation}
holds for all times on the interval of existence provided that it holds at at least one point
in the interval.

Finally, by using the compatibility condition \eqref{compatibility} in \eqref{F01-prelim} and \eqref{F02-prelim}, we
recover the $F_{0j}$ equations of \eqref{SM-gauge}.
\end{proof}

\begin{rem}
The initial data of $(\psi, A)$ may be chosen in any way that is consistent with the curvature
constraints and compatibility condition.
\end{rem}

\begin{rem}
The time compatibility conditions $D_0 \psi_k = D_k \psi_0$ are not present because we have no need for---and therefore have not introduced---the derivative field $\psi_0$.
\end{rem}

\begin{rem}
A Lagrangian approach to Schr\"odinger maps into $\Sp^2$ appears in \cite{MaPaSo94},
though the Euler-Lagrange equations there derived do not include the compatibility
condition $D_1 \psi_2 = D_2 \psi_1$. Instead, such a constraint must be imposed.
One of the key differences between our action and that introduced in \cite{MaPaSo94} is that, instead
of using a term quartic in $\psi$, we introduce a term that is quadratic in $\psi$ and linear in $dA$, which has
the effect of coupling $\psi$ and $dA$.
\end{rem}

\section{Conservation laws} \label{Sec:ConservationLaws}

Some conservation laws are written at the gauge level in \cite{MaPaSo94}
and derived at the level of maps in \cite{GrSt02}.
Our approach here is at the gauge level, in the spirit of \cite{BeIoKeTa11, CoCzLe11}.
We begin by introducing the symmetric pseudo-stress-energy tensor $T_{\alpha \beta}$, defined by
\begin{equation} \label{Tensor}
\begin{cases}
T_{00} &= \frac12 (|\psi_1|^2 + |\psi_2|^2) \\
T_{0j} &= \Im(\bar{\psi}_\ell D_j \psi_\ell) \\
T_{jk} &= 2 \Re(\overline{D_j \psi_\ell} D_k \psi_\ell) - \delta_{jk} \Delta T_{00}
\end{cases}
\end{equation}
\begin{thm} \label{thm:laws}
Solutions $(\psi, A)$ of the energy-critical gauged Schr\"odinger map system \eqref{SM-gauge} 
satisfy the conservation law
\begin{equation} \label{law1}
\partial_\alpha T_{0 \alpha} = 0
\end{equation}
and the balance law
\begin{equation} \label{law2}
\partial_\alpha T_{j \alpha} = 2 F_{\alpha j} T_{0 \alpha}
\end{equation}
\end{thm}
\begin{proof}
First we establish \eqref{law1}.
Using the evolution equations in \eqref{SM-gauge}, we have
\[
\begin{split}
\frac12 \partial_t |\psi_1|^2 =
\Re(\bar{\psi}_1 D_t \psi_1) 
&= \Re(\bar{\psi}_1 i D_j D_j \psi_1) + \Re(\bar{\psi}_1 F_{j1} \psi_j) \\
&= \partial_j \Re(\bar{\psi}_1 i D_j \psi_1) + F_{21} \Re(\bar{\psi}_1 \psi_2)
\end{split}
\]
and
\[
\frac12 \partial_t |\psi_2|^2
=
\partial_j \Re(\bar{\psi}_2 i D_j \psi_2) + F_{12} \Re(\bar{\psi}_2 \psi_1)
\]
Consequently,
\[
\frac12 \partial_t (|\psi_1|^2 + |\psi_2|^2)
=
\partial_j \Re(\bar{\psi}_\ell i D_j \psi_\ell)
\]
Next we show \eqref{law2}, which is more involved.
We start by using the evolution and curvature conditions to obtain
\begin{align}
\partial_t T_{0j} 
&=\Im(\overline{D_t \psi_\ell} D_j \psi_\ell) + \Im(\bar{\psi}_\ell D_t D_j \psi_\ell) \nonumber \\
&= \Im(\overline{i D_k D_k \psi_\ell} D_j \psi_\ell) + \Im(\overline{F_{k \ell} \psi_k} D_j \psi_\ell) +  
\Im(\bar{\psi}_\ell D_j D_t \psi_\ell) + \Im(\bar{\psi}_\ell i F_{0j} \psi_\ell) \label{tToj}
\end{align}
The rightmost term of \eqref{tToj} can be rewritten as
\[
\Im(\bar{\psi}_\ell i F_{0j} \psi_\ell) = F_{0j}(|\psi_1|^2 + |\psi_2|^2) = 2 F_{0j} T_{00}
\]
In view of the evolution equation, curvature conditions, and compatibility condition,
the second-to-last term of \eqref{tToj} expands as
\begin{align}
\Im(\bar{\psi}_\ell D_j D_t \psi_\ell) 
&= \Im(\bar{\psi}_\ell i D_j D_k D_k \psi_\ell) + \Im(\bar{\psi}_\ell D_j(F_{k \ell} \psi_k)) \nonumber \\
&= \Im(\bar{\psi}_\ell i D_k D_j D_k \psi_\ell) - \Im(\bar{\psi}_\ell F_{jk} D_k \psi_\ell) + \Im(\bar{\psi}_\ell D_j(F_{k \ell} \psi_k)) \nonumber \\
&= \Im(\bar{\psi}_\ell i D_k D_j D_k \psi_\ell) + \Im(\bar{\psi}_\ell D_j(F_{k \ell} \psi_k)) + F_{kj} T_{0k} \nonumber \\
&= \partial_k \Im(\bar{\psi}_\ell i D_j D_k \psi_\ell) - \Im(\overline{D_k \psi_\ell} i D_j D_k \psi_\ell)
+ \Im(\bar{\psi}_\ell D_j(F_{k \ell} \psi_k)) + F_{kj} T_{0k} \label{take1}
\end{align}
Appealing only to the curvature conditions and compatibility condition, we rewrite
the first term of \eqref{tToj} as
\[
\begin{split}
\Im(\overline{i D_k D_k \psi_\ell} D_j \psi_\ell) 
&= \partial_k \Im(\overline{i D_k \psi_\ell} D_j \psi_\ell) - \Im(\overline{i D_k \psi_\ell} D_k D_j \psi_\ell) \\
&= \partial_k \Im(\overline{i D_\ell \psi_k} D_j \psi_\ell) - \Im(\overline{i D_\ell \psi_k} D_k D_j \psi_\ell)
\end{split}
\]
where we then rewrite $-\Im(\overline{i D_\ell \psi_k} D_k D_j \psi_\ell)$ as
\[
\begin{split}
-\Im(\overline{i D_\ell \psi_k} D_k D_j \psi_\ell)
&= - \Im(\overline{i D_\ell \psi_k} D_j D_k \psi_\ell) - \Im(\overline{i D_\ell \psi_k} i F_{kj} \psi_\ell) \\
&= - \Im(\overline{i D_\ell \psi_k} D_j D_k \psi_\ell) - F_{kj} \Im(\overline{D_k \psi_\ell} \psi_\ell) \\
&= - \Im(\overline{i D_\ell \psi_k} D_j D_k \psi_\ell) + F_{kj} T_{0k}
\end{split}
\]
so that
\begin{equation} \label{take2}
\Im(\overline{i D_k D_k \psi_\ell} D_j \psi_\ell) = 
\partial_k \Im(\overline{i D_\ell \psi_k} D_j \psi_\ell) - \Im(\overline{i D_\ell \psi_k} D_j D_k \psi_\ell) + F_{kj} T_{0k}
\end{equation}
Taking \eqref{take1} and \eqref{take2} together, we get
\[
\begin{split}
&\Im(\bar{\psi}_\ell D_j D_t \psi_\ell) + \Im(\overline{i D_k D_k \psi_\ell} D_j \psi_\ell)  \\
&\quad = \partial_k \Im(\bar{\psi}_\ell i D_j D_k \psi_\ell) +  \partial_k \Im(\overline{i D_\ell \psi_k} D_j \psi_\ell)
+ \Im(\bar{\psi}_\ell D_j(F_{k \ell} \psi_k)) + 2 F_{kj} T_{0k}
\end{split}
\]
Therefore
\begin{equation} \label{progress}
\begin{split}
\partial_t T_{0j} &= 2 F_{\alpha j} T_{0 \alpha}  \\
&\quad + \partial_k \Im(\bar{\psi}_\ell i D_j D_k \psi_\ell) +  \partial_k \Im(\overline{i D_\ell \psi_k} D_j \psi_\ell) +
\Im(\bar{\psi}_\ell D_j(F_{k \ell} \psi_k)) + \Im(\overline{F_{k \ell} \psi_k} D_j \psi_\ell)
\end{split}
\end{equation}
The last line of \eqref{progress} may be rewritten as
\begin{equation} \label{pent}
\partial_k \partial_j \Im(\bar{\psi}_\ell i D_k \psi_\ell) - 2 \partial_k \Im(\overline{D_j \psi_\ell} i D_k \psi_\ell)
+ \partial_j \Im(\bar{\psi}_\ell F_{k \ell} \psi_k) - 2 \Im(\overline{D_j \psi_\ell} F_{k\ell} \psi_k)
\end{equation}
The first term of \eqref{pent} can be rewritten as $\partial_j \Delta T_{00}$. The third term of \eqref{pent} is $\mu \partial_j F_{12}^2$.
For the fourth term, we have
\[
- 2 \Im(\overline{D_j \psi_\ell} F_{k\ell} \psi_k) = - 2 F_{12} \mu \partial_j F_{12} = - \mu \partial_j F_{12}^2
\]
Therefore we may rewrite \eqref{progress} as follows:
\[
\partial_t T_{0j} = 
- 2 \partial_k \Im(\overline{D_j \psi_\ell} i D_k \psi_\ell)
+ \partial_j \Delta T_{00} + 2 F_{\alpha j} T_{0 \alpha}
\]
\end{proof}

\begin{cor}
For rapidly decaying solutions of \eqref{SM-gauge}, the following quantity is conserved:
\begin{equation} \label{g-en}
\frac12 \int_{\R^2} \left( |\psi_1|^2 + |\psi_2|^2\right) dx
\end{equation}
\end{cor}
Note that \eqref{g-en} is simply \eqref{Lag} written at the level of frames.
\begin{lem}[Hamiltonian]
Let
\begin{equation} \label{Hamiltonian}
H_{Sch} := \int_{\R^2} \left( -\Im(\overline{D_j \psi_2} D_j \psi_1) + \frac12 (|\psi_1|^2 + |\psi_2|^2) F_{12}  \right) dx^1 \wedge dx^2
\end{equation}
Then, for rapidly decaying solutions of the gauged Schr\"odinger map system 
\eqref{SM-gauge}, it holds that
\[
H_{Sch} = \frac12 \int_{\R^2} \left( \partial_1 T_{02} - \partial_2 T_{01} \right) dx^1 \wedge dx^2 = 0
\]
\end{lem}
\begin{proof}
Using the compatibility and curvature conditions, we calculate
\begin{equation} \label{string}
\begin{split}
\Im(\overline{D_1 \psi_2} D_1 \psi_1)
&=
\Im(\overline{D_2 \psi_1} D_1 \psi_1) \\
&=
\partial_1 \Im(\overline{D_2 \psi_1} \psi_1) - \Im(\overline{D_1 D_2 \psi_1} \psi_1) \\
&=
\partial_1 \Im(\overline{D_2 \psi_1} \psi_1) - \Im(\overline{D_2 D_1 \psi_1} \psi_1) - \Im(\overline{i F_{12} \psi_1} \psi_1)
\end{split}
\end{equation}
The right hand side we may expand as
\[
\partial_1 \Im(\overline{D_2 \psi_1} \psi_1) - \partial_2 \Im(\overline{D_1 \psi_1} \psi_1) + \Im(\overline{D_1 \psi_1} D_2 \psi_1) + F_{12} |\psi_1|^2
\]
which, by virtue of the string of equalities in \eqref{string}, is equal to $\Im(\overline{D_2 \psi_1} D_1 \psi_1)$.
This implies
\[
2 \Im(\overline{D_2 \psi_1} D_1 \psi_1)
=
\partial_1 \Im(\overline{D_2 \psi_1} \psi_1) - \partial_2 \Im(\overline{D_1 \psi_1} \psi_1) + F_{12} |\psi_1|^2
\]
and hence
\[
\Im(\overline{D_1 \psi_2} D_1 \psi_1) = \frac12 \left( \partial_1 \Im(\overline{D_2 \psi_1} \psi_1) - \partial_2 \Im(\overline{D_1 \psi_1} \psi_1) + F_{12} |\psi_1|^2 \right)
\]
By conjugating and reversing the roles of the indices, we similarly conclude
\[
\Im(\overline{D_2 \psi_2} D_2 \psi_1) = \frac12 \left(\partial_2 \Im(\bar{\psi}_2 D_1 \psi_2) - \partial_1 \Im(\bar{\psi}_2 D_2 \psi_2) + F_{12} |\psi_2|^2 \right)
\]
Therefore
\[
-\Im(\overline{D_j \psi_2} D_j \psi_1) + \frac12 (|\psi_1|^2 + |\psi_2|^2) F_{12}
=
\frac12 \left( \partial_1 \Im(\bar{\psi}_j D_2 \psi_j) - \partial_2 \Im(\bar{\psi}_j D_1 \psi_j) \right)
\]
\end{proof}
Define the virial potential and Morawetz action respectively by
\[
V_a(t) = \int_{\R^2} a(x) T_{00} dx,
\quad \quad
M_a(t) = \int_{\R^2} T_{0 j} \partial_j a  \;dx
\]
The conservation law \eqref{law1} followed by integration by parts implies
\[
\partial_t V_a(t) = M_a(t)
\]
We recover the generalized virial identity of \cite[Lemma 3.1]{CoCzLe11}, adapted
to the setting of Schr\"odinger maps.
\begin{lem}
Let $a:\R^2 \to \R$ and let $(\psi, A)$ be a solution of \eqref{SM-gauge}. Then
\[
M_a(T) - M_a(0) =
\int_0^T \int_{\R^2} \left[ 2 \Re(\overline{D_j \psi_\ell} D_k \psi_\ell) \partial_j \partial_k a  
- \delta_{jk} T_{00} \Delta^2 a + 2 F_{\alpha j} T_{0\alpha} \partial_j a \right] dx dt
\]
\end{lem}
\begin{proof}
Using the Morawetz action, balance law, and integration by parts, we have
\[
\partial_t M_a(t) = 2 \int_{\R^2} \left( T_{jk} \partial_k \partial_j a + F_{\alpha j} T_{0\alpha} \partial_j a \right) dx
\]
\end{proof}
\begin{cor}
If $a$ is convex, then we can further conclude that
\[
\int_0^T \int_{\R^2} \left( 2 F_{\alpha j} T_{0 \alpha} \partial_j a - \delta_{jk} T_{00} \Delta^2 a \right)  dx dt \lesssim \sup_{[0, T]} |M_a(t)|
\]
\end{cor}
Virial identities are established in the context of equivariant Schr\"odinger maps in \cite{BeIoKeTa11, BeIoKeTa12}.
For virial and Morawetz identities in the context of radial Schr\"odinger maps, see \cite{GuKo11}.

\section{Comparison with Chern-Simons-Schr\"odinger} \label{Sec:CSS}

In two spatial dimensions, the Chern-Simons-Schr\"odinger equation
arises as the second-quantization of a nonrelativistic anyon system.
For background, see \cite{JaTe81, DeJaTe82, Wi90, EzHoIw91, EzHoIw91b, JaPi91b, JaPi91, JaPiWe91, MaPaSo91}.
Local wellposedness at high regularity is established in \cite{BeBoSa95} using the Coulomb
gauge and at low-regularity for small data in \cite{LiSmTa12} using the heat gauge, which,
in the setting of Chern-Simons-Schr\"odinger systems,
appears to have been first introduced in \cite{De07}. Global wellposedness
for large data under an equivariance ansatz and the Coulomb gauge is established at the critical regularity
of $L^2$ in \cite{LiSm13}.

\noindent \textbf{Lagrangian formulation.}
The action is
\[
L(\phi, A) = \frac12 \int_{\R^{2+1}}  
\left[ \Im (\bar \phi D_t \phi) + |D_x \phi|^2 - \frac{g}{2} |\phi|^4 \right] dx^1 \wedge dx^2 \wedge dt +
\frac12 \int_{\R^{2+1}} A \wedge dA
\]
with Euler-Lagrange equations
\begin{equation} \label{CSS}
\begin{cases}
D_t \phi &= i D_\ell D_\ell \phi + i g \lvert \phi \rvert^2 \phi \\
F_{01} &= - \Im(\bar{\phi} D_2 \phi) \\
F_{02} &= \Im(\bar{\phi} D_1 \phi) \\
F_{12} &= -\frac{1}{2} \lvert \phi \rvert^2
\end{cases}
\end{equation}
both of which enjoy the gauge freedom \eqref{gauge-freedom}.
It is interesting to note that \eqref{CSS} is Galilean invariant,
whereas \eqref{SM-gauge} is not; the obstruction lies with the compatibility condition.
In the above $g$ is a coupling constant. The so-called ``critical coupling"
is $g = 1$, and this is what we consider below.

\noindent \textbf{Conservation laws.}
For the Chern-Simons-Schr\"odinger system \eqref{CSS}, we set, following \cite{CoCzLe11}, 
\[
\begin{cases}
T_{00} &= \frac12 |\phi|^2 \\
T_{0j} &= \Im(\bar{\phi} D_j \phi) \\
T_{jk} &= 2 \Re(\overline{D_j \phi} D_k \phi) - \delta_{jk}( T_{00} + \Delta) T_{00}
\end{cases}
\]
Here there is no distinction between the conservation law \eqref{law1} 
and the curvature relation \eqref{geocon}.
Note that for $\phi$ not identically zero we always have
\[
\int_{\R^2} d\uA = \int_{\R^2} F_{12} dx = - \int_{\R^2} T_{00} dx < 0
\]
The balance law \eqref{law2} is still valid in this context.
Its right hand side, however, vanishes thanks to
$F_{01} = - T_{02}$, $F_{02} = T_{01}$, and $F_{12} = - T_{00}$, so that
\begin{equation} \label{CSS-conserv}
\partial_\alpha T_{j \alpha} = 0
\end{equation}
The conserved energy for this system is
\[
E(\phi) := \frac12 \int_{\R^2} \left( |D_x \phi|^2 - \frac12 |\phi|^4 \right) dx
=
\frac14 \int_{\R^2} \left( T_{11} + T_{22} \right) dx
\]

\noindent \textbf{Virial identities.}
In spite of \eqref{CSS-conserv}, the focusing nature of the critical coupling adds
a term in the generalized virial identity with a sign that is unfavorable for establishing Morawetz estimates.
In particular, this term is the $-\delta_{jk} T_{00}^2$ term appearing in the definition of $T_{jk}$.
\begin{lem}
Let $a: \R^2 \to \R$ and let $(\phi, A)$ be a solution of \eqref{CSS}. Then
the Morawetz action $M_a(t)$
satisfies
\begin{equation} \label{Mora}
M_a(T) - M_a(0) =
\int_0^T \int_{\R^2}
\left[ 2 \Re(\overline{D_j \phi} D_k \phi) \partial_j \partial_k a - \delta_{jk} \left(T_{00} \Delta^2 a + T_{00}^2 \Delta a\right) \right] dx dt
\end{equation}
\end{lem}
\begin{cor}
For $a = |x|^2$, it holds that
\begin{equation} \label{CSS-virial}
\partial_t^2 \int_{\R^2} |x|^2 T_{00} dx = 
\partial_t M_{\{a = |x|^2\}}(t) =
2 \int_{\R^2} \left( |D_x \phi|^2 - 2 T_{00}^2 \right) dx = 4 E(\phi)
\end{equation}
\end{cor}
Equation \eqref{CSS-virial} was used in \cite{BeBoSa95} to establish the existence of finite-time blow-up solutions by taking data
with negative energy or data with positive energy and sufficiently large weighted momentum.
We remark that \cite{Hu09} constructs finite-time blow-up solutions that have charge equal to that of the ground state. 
The key tool in the construction is pseudo-conformal invariance.
The fact that \eqref{CSS-virial} holds is closely tied with exact conservation laws and pseudo-conformal invariance
\cite[\S 2.4]{Tao06}.
In the case of Schr\"odinger maps, \eqref{law2} is not an exact conservation law.
Moreover, pseudo-conformal invariance fails to hold, the obstruction being the compatibility condition \cite{Hu08}.
If the compatibility condition were dropped, then $H_{Sch}$ introduced in \eqref{Hamiltonian} could be made
to be nonzero, but not otherwise for maps with sufficient decay.
Constructing blow-up solutions for Schr\"odinger maps is therefore more involved \cite{MeRaRo11, MeRaRo11a, Pe12};
see also the complementary stability result \cite{BeTa10}.

\textbf{Acknowledgments}

The author thanks the anonymous referee for helpful stylistic suggestions and minor corrections.

\appendix

\section{Gradient flow and solitons}

\begin{thm}
The energy-critical gauged harmonic map heat flow system \eqref{HMHF-gauge} is generated
by the gradient flow of
\begin{equation} \label{Hhar}
H_{Har}(\psi, A) :=
\frac12 \int_{\R^2} \left(\Re(\overline{D_j \psi_k} D_j \psi_k) -
 \mu [\Im(\bar{\psi}_2 \psi_1)]^2 \right) dx^1 \wedge dx^2 
 - \mu \frac12 \int_{\R^2} d\uA \wedge \ast d\uA
\end{equation}
provided
\begin{equation} \label{gradcon}
F_{12} = \mu \Im(\bar{\psi}_2 \psi_1),
\quad \quad
D_1 \psi_2 = D_2 \psi_1
\end{equation}
at the initial time.
\end{thm}
\begin{proof}
We first obtain \eqref{HMHF-gauge} with $A_0 = 0$. A posteriori one may incorporate
$A_0 \neq 0$ in the flow and retain gauge invariance. Note that \eqref{Hhar} itself is invariant
under time-independent gauge transformations.

Varying $\psi_1$ leads to
\[
- \int_{\R^2} \left[ \Re(\bar{\phi} D_j D_j \psi_1) + \mu \Im(\bar{\psi}_2 \psi_1)\Im(\bar{\psi}_2 \phi)
\right] dx^1 \wedge dx^2
\]
The associated downward gradient flow for $\psi_1$ is therefore
\begin{equation} \label{psik-flow1}
\partial_t \psi_1 = D_j D_j \psi_1 + i \mu \Im(\bar{\psi}_2 \psi_1) \psi_2
\end{equation}
Similarly, varying $\psi_2$ leads to
\begin{equation} \label{psik-flow2}
\partial_t \psi_2 = D_j D_j \psi_2 - i \mu \Im(\bar{\psi}_2 \psi_1) \psi_1
\end{equation}
Varying $A_j$ leads to
\[
B_j \Im(\bar{\psi}_k D_j \psi_k)  - B_j \mu \partial_k F_{jk}
\]
Which gradient direction we choose for $A_j$ depends upon $\mu$:
when we couple the $F_{0j}$ equations with \eqref{geocon},
we choose the sign so as to obtain a forward heat evolution for $F_{12}$ rather than a backward one.
Therefore we take
\begin{equation} \label{Fs}
F_{0j} = \partial_t A_j = \mu \Im(\bar{\psi}_k D_j \psi_k) - \partial_k F_{jk}
\end{equation}
so that coupling \eqref{Fs} with \eqref{geocon} yields
\begin{equation} \label{compare1}
\begin{split}
(\partial_t - \Delta) F_{12} &= 
\mu \left[ \partial_1 \Im(\bar{\psi}_k D_2 \psi_k) -
\partial_2 \Im(\bar{\psi}_k D_1 \psi_k) \right]  \\
&= \mu \left[ -2 \Im(\overline{D_2 \psi_k} D_1 \psi_k) + \Im(\bar{\psi}_k (D_1 D_2 - D_2 D_1) \psi_k) \right] \\
&= \mu \left[ -2 \Im(\overline{D_2 \psi_k} D_1 \psi_k) + F_{12} (|\psi_1|^2 + |\psi_2|^2) \right]
\end{split}
\end{equation}
On the other hand, using \eqref{psik-flow1}, \eqref{psik-flow2}, we get
\[
\begin{split}
\partial_t \Im(\bar{\psi}_2 \psi_1) &=
\Im(\bar{\psi}_2 \partial_t \psi_1) - \Im(\bar{\psi}_1 \partial_t \psi_2) \\
&= \Im(\bar{\psi}_2 D_j D_j \psi_1) - \Im(\bar{\psi}_1 D_j D_j \psi_2) + \mu \Im(\bar{\psi}_2 \psi_1)(|\psi_1|^2 + |\psi_2|^2) \\
&= \Delta \Im(\bar{\psi}_2 \psi_1) - 2 \Im(\overline{D_j \psi_2} D_j \psi_1) + \mu \Im(\bar{\psi}_2 \psi_1)(|\psi_1|^2 + |\psi_2|^2)
\end{split}
\]
so that
\begin{equation} \label{compare2}
(\partial_t - \Delta) \mu \Im(\bar{\psi}_2 \psi_1) = - 2 \mu \Im(\overline{D_j \psi_2} D_j \psi_1) + \Im(\bar{\psi}_2 \psi_1)(|\psi_1|^2 + |\psi_2|^2)
\end{equation}
Comparing \eqref{compare1} and \eqref{compare2}
suggests $F_{12} = \mu \Im(\bar{\psi}_2 \psi_1)$ and $D_1 \psi_2 = D_2 \psi_1$, and
we enforce this at the initial time.
Using a Gronwall inequality argument similar to that in Theorem \ref{thm:SchLag},
we conclude that these constraints persist forward in time.
In particular, set
\[
\Theta := D_1 \psi_2 - D_2 \psi_1,
\quad \quad
\Psi := F_{12} - \mu \Im(\bar{\psi}_2 \psi_1)
\]
One the one hand, we can use \eqref{psik-flow1}, \eqref{psik-flow2}, and \eqref{Fs} to obtain
\begin{equation} \label{Gron1}
(D_t - D_j D_j) \Theta = 2 i \Psi D_j \psi_j + \mu i \psi_j \Im(\bar{\psi_j} \Theta)
\end{equation}
On the other hand, subtracting \eqref{compare2} from \eqref{compare1} leads to
\begin{equation} \label{Gron2}
(\partial_t - \Delta) \Psi = 2 \mu \Im(\bar{\Theta} D_j \psi_j) + \mu \Psi (|\psi_1|^2 + |\psi_2|^2)
\end{equation}
Together \eqref{Gron1} and \eqref{Gron2} suffice for controlling  $\int_{\R^2} ( |\Theta|^2 + |\Psi|^2 )(t) dx$
for sufficiently regular $(\psi, A)$.
\end{proof}

\begin{rem}
One may ask whether
\[
\frac12 \int_{\R^2} \left(\Re(\overline{D_j \psi_k} D_j \psi_k) -
 \mu [\Im(\bar{\psi}_2 \psi_1)]^2 \right) dx^1 \wedge dx^2 
\]
of \eqref{Hhar} is conserved by the Schr\"odinger flow \eqref{SM-gauge}.
A straightforward calculation reveals
\[
\frac12 \partial_t 
\int_{\R^2} \left( \Re(\overline{D_j \psi_k} D_j \psi_k) - \mu [\Im(\bar{\psi}_2 \psi_1)]^2 \right) dx
=
\int F_{0j} T_{0j} dx
\]
One may rewrite the $F_{0j} T_{0j}$ term in various ways using
\[
\mu F_{0j} = \partial_k \Re(\bar{\psi}_j \psi_k) - \partial_j T_{00},
\quad
T_{0j} = \mu \partial_\ell F_{j \ell} + \Im(\bar{\psi}_j D_\ell \psi_\ell),
\]
though it does not appear that this term can be expected to vanish
unless $(\psi, A)$ satisfy the gauged harmonic map equations \eqref{HM-gauge}.
\end{rem}
We conclude our discussion of harmonic map heat flow by noting
that it is the main tool used in the construction of the caloric gauge (see \cite{Tao04, Sm12}).

We now turn to solitons, by which we mean steady states.
Evident from the geometric map formulation of the evolution equations is that harmonic maps
constitute the solitons for both harmonic map heat flow and Schr\"odinger maps.

\begin{lem} \label{lem:soliton}
If $(\psi, A)$ are smooth and satisfy one of the two systems
\begin{equation} \label{self-dual}
\begin{cases}
(D_1 \pm i D_2) \psi_1 &= 0 \\
(D_1 \pm i D_2) \psi_2 &= 0 \\
F_{12} &= \mu \Im(\bar{\psi}_2 \psi_1) \\
D_1 \psi_2 &= D_2 \psi_1
\end{cases}
\end{equation}
and $F_{12} \neq 0$ at at least one point,
then $(\psi, A)$ satisfy the energy-critical gauged harmonic map system \eqref{HM-gauge}.
Moreover, if we take $A_0 = 0$, then such $(\psi, A)$ also
provide stationary solutions of the gauged harmonic map
heat flow system \eqref{HMHF-gauge}
and the gauged Schr\"odinger map system \eqref{SM-gauge}.
\end{lem}
\begin{proof}
We consider solutions of \eqref{self-dual} with $``+"$, as solutions of \eqref{self-dual} with $``-"$ may be handled
similarly.
Invoking the compatibility condition $D_1 \psi_2 = D_2 \psi_1$, 
we conclude from the first two equations of \eqref{self-dual} that
\begin{equation} \label{CR}
\begin{cases}
D_1(\psi_1 + i \psi_2) &= 0 \\
D_2(\psi_1 + i \psi_2) &= 0
\end{cases}
\end{equation}
Differentiating \eqref{CR} leads to
\[
(D_1 D_2 - D_2 D_1) (\psi_1 + i \psi_2) = i F_{12} (\psi_1 + i \psi_2) = 0
\]
which in view of the fact that $F_{12} \neq 0$
at at least one point, implies that
\begin{equation} \label{harmonic}
\psi_1 = - i \psi_2
\end{equation}
holds at some $x \in \R^2$.
Thanks to \eqref{CR}, we have for $j = 1, 2$ that
\[
\partial_j |\psi_1 + i \psi_2|^2 = 0
\]
for all $x \in \R^2$. Therefore \eqref{harmonic} also holds for all $x \in \R^2$.
Together \eqref{harmonic} and \eqref{self-dual} imply
\[
D_j \psi_j = (D_1 + i D_2) \psi_1 = 0
\]
To see that \eqref{HMHF-gauge} and \eqref{SM-gauge} are satisfied, it only remains to use
the compatibility and spatial curvature conditions to write
\[
D_j D_j \psi_k - i F_{jk} \psi_j = D_j D_k \psi_j - i F_{jk} \psi_j = D_k D_j \psi_j = 0
\]
\end{proof}

\noindent \textbf{Comparison with Chern-Simons-Schr\"odinger.}
The gradient flow of
\[
\frac12 \int_{\R^2} \left[ |D_x \phi|^2 - \frac12 |\phi|^4 \right] dx^1 \wedge dx^2
+ \frac12 \int_{\R^2} d\uA \wedge \ast d\uA
\]
yields
\[
\begin{cases}
D_t \phi &= D_\ell D_\ell \phi + |\phi|^2 \phi \\
F_{01} &= -\Im(\bar{\phi} D_1 \phi) - \partial_2 F_{12} \\
F_{02} &= -\Im(\bar{\phi} D_2 \phi) + \partial_1 F_{12}
\end{cases}
\]
which with $d^2 A = 0$ imply 
\begin{equation} \label{F12-gradflow}
(\partial_t - \Delta) F_{12} = \partial_1 \Im(\bar{\phi} D_2 \phi) - \partial_2 \Im(\bar{\phi} D_1 \phi)
\end{equation}
In analogy with the caloric gauge for Schr\"odinger maps, the Chern-Simons gradient flow
can be used to construct a caloric gauge for the Chern-Simons-Schr\"odinger system.
Doing so requires initializing \eqref{F12-gradflow} with suitable data at time zero.
Both caloric gauges can be interpreted as modifications of the Coulomb gauge.
In the case of Chern-Simons-Schr\"odinger, the modification is not sufficient to render the caloric
gauge a favorable alternative to the Coulomb gauge for the purposes of establishing wellposedness.
Behind these differences in behavior is the difference between the $F_{12}$ curvature terms.

Stationary solutions of \eqref{CSS} are the so-called self-dual and anti-self-dual Chern-Simons solitons,
which are solutions of
\[
\begin{cases}
(D_1 \pm i D_2) \phi &= 0 \\
F_{01} &= - \Im(\bar{\phi} D_2 \phi) \\
F_{02} &= \Im(\bar{\phi} D_1 \phi) \\
F_{12} &= - \frac12 |\phi|^2 \\
\end{cases}
\]
Note that $A_0$ is not identically zero so long as $\phi$ is not. We refer the reader to \cite{HuSe13, DuJaPi91, DuJaPi92}.

\bibliographystyle{amsplain}

\end{document}